\documentclass[12pt]{amsart}
\usepackage{amssymb}
\usepackage{amsmath, amscd}
\usepackage{amsthm}
\usepackage{xcolor}

\newtheorem{theorem}{Theorem}[section]

\newtheorem{proposition}[theorem]{Proposition}
\newtheorem{lemma}[theorem]{Lemma}

\newtheorem{corollary}[theorem]{Corollary}
\theoremstyle{definition}

\newtheorem{example}[theorem]{Example}
\newtheorem{definition}[theorem]{Definition}

\newtheorem{remark}[theorem]{Remark}

\def\a{\alpha}

\def\t{\tau}

\def\val#1{\vert #1 \vert}

\def\bbvec#1{{\bf #1}}

\usepackage[utf8]{inputenc}
\usepackage{amssymb}
\usepackage{amsmath, amscd}
\usepackage{amsthm}
\usepackage{xcolor}

\begin{document}

\author[A.R. Chekhlov]{Andrey R. Chekhlov}
\address{Department of Mathematics and Mechanics, Tomsk State University, 634050 Tomsk, Russia}
\email{cheklov@math.tsu.ru; a.r.che@yandex.ru}
\author[P.V. Danchev]{Peter V. Danchev}
\address{Institute of Mathematics and Informatics, Bulgarian Academy of Sciences, 1113 Sofia, Bulgaria}
\email{danchev@math.bas.bg; pvdanchev@yahoo.com}
\author[P.W. Keef]{Patrick W. Keef}
\address{Department of Mathematics, Whitman College, Walla Walla, WA 99362, USA}
\email{keef@whitman.edu}

\title[Transitivity-like Properties and Torsion-Freeness] {On Transitivity-Like Properties for \\ Torsion-Free Abelian Groups}

\maketitle

\begin{abstract} We study some close relationships between the classes of transitive, fully transitive and Krylov transitive torsion-free Abelian groups. In addition, as an application of the achieved assertions, we resolve some old-standing problems, posed by Krylov-Mikhalev-Tuganbaev in their monograph \cite{KMT}. Specifically, we answer Problem 44 from there in the affirmative by constructing a Krylov transitive torsion-free Abelian group which is neither fully transitive nor transitive. This extends to the torsion-free case certain similar results in the $p$-torsion case, obtained by Braun et al. in J. Algebra (2019).

We, alternatively, also expand to the torsion-free version some of the results concerning transitivity, full transitivity and Krylov transitivity in the $p$-primary case due Files-Goldsmith from Proc. Amer. Math. Soc. (1998) and Danchev-Goldsmith from J. Comm. Algebra (2011).
\end{abstract}

\vskip2.0pc

\section{Introduction and Fundamentals}
	
Throughout the rest of the paper, unless specified something else, all groups will be additively written torsion-free Abelian groups. We will primarily use the notation and terminology of \cite{F0,F1,F2}, but we will follow somewhat those from \cite{K} and \cite{Gr} as well. We begin with a quick review of some of the most important.

If $a$ is an element of the group $G$ and $p$ is a prime, we denote the $p$-height of $a$ by $\val a_p$, or perhaps either by $\val a_{G,p}$ or $\val a_G$, if we wish to specify the context in which group the height is calculated. By a {\it characteristic} we mean a sequence $\chi=(\chi_1,\chi_2, \dots)$ of non-negative integers or the symbol $\infty$. In particular, if $\mathcal{P}:=\{p_1, p_2, \dots\}$ is the set of all primes, then the height sequence of $a\in G$, $\chi(a)=\chi_G(a)=(\val {a}_{p_1},\val {a}_{p_2}, \dots), $ is a characteristic.

For a characteristic $\chi$ there is clearly a unique subgroup ${\mathbb Q}_\chi\subseteq \mathbb Q $ such that $1\in {\mathbb Q}_\chi$ and $\chi_{{\mathbb Q}_\chi}(1)=\chi$. And if $\bf e$ is an element of some rank-1 group $A$, and $\chi= \chi_A({\bf e})$, then we often denote $A$ by ${\mathbb Q}_\chi{\bf e}:=\{q{\bf e}:q\in {\mathbb Q}_\chi\}$. For the group $G$ and characteristic $\chi$, we let $p^\chi G=\{a\in G: \chi_G(a)\geq \chi\}$. Clearly, $a\in p^\chi G$ if, and only if, ${\mathbb Q}_\chi a\subseteq G$. For shortness and to avoid some possible confusion, we shall hereafter just write $G(\chi)$ instead of $p^\chi G$, thus paralleling it with the stated below notation $G(\tau)$.

If $\chi$ is a characteristic and $n\in \mathbb Z$, then we let $n\chi$ be the characteristic such that for all $i=1, 2, \dots$, $(n\chi)_i=\chi_i+\val n_{{\mathbb Z},p_i}$. In particular, if $a\in G$, then $\chi_G(na)=n\chi_G(a)$. The characteristics $\chi$, $\chi'$ are equivalent if there are non-zero $n,n'\in \mathbb Z$ such that $n\chi=n'\chi'$. It is easy to check this is an equivalence relation. A resulting equivalence class is called a {\it type}; we typically denote this by $\tau=\overline \chi$.

If $a\in G$, then we write $\tau(a)=\overline {\chi}(a)$. The {\it type set} of $G$, written $\t(G)=\{\t(a) ~ | ~ 0\neq a\in G\}$, is the set of types of all non-zero elements of $G$. We say $G$ is {\it homogeneous of type $\tau$}, or simply {\it ${\tau}$-homogeneous}, if $\t(G)=\{\tau\}$.

The natural ordering of characteristics leads to a natural ordering of types. Clearly, for a type $\t=\overline \chi$ and group $G$, the set $G(\t)=\{a\in G ~|~ \t(a)\geqslant \t\}$ is a pure fully invariant subgroup of  $G$. It is easy to see that $G(\tau)=G$ if, and only if, the quotient $G/G(\chi)$ is torsion.

Finally, the group $G$ is {\it separable} if every finite subset of $G$ is contained in a summand $C$ which is a finite rank completely decomposable group. If  $G$ is separable, then $\Omega(G)\subseteq \t(G)$ stands for the set of types of all rank-$1$ direct summands of $G$. Besides, if $G$ is arbitrary, then we set $\pi(G)=\{p \in {\mathcal P} ~|~ pG\neq G\}$. And if $\tau$ is type, then we put $\pi(\tau)=\pi(G)$, where $G$ is a rank-$1$ group of type $\tau$.

\medskip

In his famous book \cite{K}, Irving Kaplansky introduced two major properties of groups as follows:

\begin{definition}\label{Kaplansky} A group $G$ is called {\it transitive} if, for any two elements $x,y$ with $\chi_G(x)=\chi_G(y)$, there exists an automorphism of $G$ mapping $x$ to $y$, and {\it fully transitive} if, for any two elements $x,y$ with $\chi_G(x)\leq \chi_G(y)$, there exists an endomorphism $\phi$ of $G$ with $\phi(x)=y$.
\end{definition}

It is known that these two notions are independent for torsion-free groups, i.e., there is a group that is transitive, but not fully transitive, and a group that is fully transitive, but not transitive (see page 385 of \cite{KMT} and the references listed there as well as \cite{Grin}). Likewise, in \cite[pp. 475-476]{KMT} is announced that Kaplansky actually defined (full) transitivity in a more global setting, namely for modules over a complete domain of discrete normalization and thus, in particular, for $p$-primary groups over some prime $p$.

Both notions of transitivity and full transitivity in Definition~\ref{Kaplansky} were extended by Krylov in \cite{Kr} to the following concept (note that Krylov used another terminology, whereas the term "Krylov transitive" was utilized for a first time by Goldsmith-Str\"ungmann -- see, e.g., \cite{BGGS}):

\begin{definition}\label{Krylov} A group $G$ is called {\it Krylov transitive} if, for any elements $x,y\in G$ with $\chi_G(x)=\chi_G(y)$, there exists an endomorphism of $G$ mapping $x$ to $y$.
\end{definition}

Clearly, a group that is transitive, but not fully transitive, will also be Krylov transitive, but not fully transition. Similarly, a group that is fully transitive, but not transitive, will also be Krylov transitive, but not transitive.

\medskip

Mimicking \cite{GS}, we state the following important concept:

\begin{definition} A group $G$ is called {\it weakly transitive} if, for any pair of elements $x, y \in G$ and endomorphisms $\phi, \psi$ of $G$ such that $\phi(x)=y$, $\psi(y)=x$, there exists an automorphism $\eta$ of $G$ with $\eta(x)=y$.
\end{definition}

To avoid any unexpected confusion with the used terminology in our further considerations, let us specify that the notion of weak transitivity given in \cite{KMT} is just the notion of Krylov transitivity stated exactly as above in Definition~\ref{Krylov}.

\medskip

The main motivation in writing up the present article is to initiate the comprehensive study of the existing relationships between all versions of the transitive properties of torsion-free Abelian groups, which examination was started in \cite{CDK}, and to compare them with the well-studied $p$-torsion case, where $p$ is an arbitrary but fixed prime (for the latter case we refer the interested readers to \cite{BGGS}, \cite{DG}, \cite{FG} and \cite{Grin}, respectively). As it will be demonstrated in the sequel, the situation here is quite more difficult and complicated than it could be anticipated as well as there is no a satisfactory symmetry with the results obtained in the primary case.

Our work is organized like this: We have provided above the reader with all background material needed for the successful reading of the text. In the next second section, we formulate our main results and give their proofs. In the third section, we construct some concrete examples, whose exhibition explains to what extent the definitions of the different variants of transitivity are either close or distinguish each other. We conclude in the final fourth section with certain comments and a list of several open problems, which naturally arise for a future work on the explored subject.

\section{Results and Proofs}

At the beginning of this section we start with Problem 41(2) from \cite{KMT} which asked the following: Under which conditions on the torsion-free transitive group $G$, the group $G^{(\mathbf{m})}=\oplus_{\mathbf{m}} G$ will remain transitive?

\medskip

We have the following simple but curious result, which strengthens the classical fact that the direct summand of a Krylov transitive group is again a Krylov transitive group.

\medskip

\begin{lemma}\label{1} If the torsion-free group $G^{(\mathfrak{m})}$, with finite or infinite ordinal number $\mathfrak{m}>1$, is Krylov transitive, then $G$ is fully transitive.
\end{lemma}

\begin{proof} Let us write $A=G_1\oplus G_2$, where $G_1\cong G$, $G_2\cong G^{(\mathfrak{m-1})}$, $x,y\in G_1$ and $|x|_G\leq|y|_G$. Observe elementarily that there exists $y'\in G_2$ with $|y'|_G=|y|_G$. Since one has that $|x+y'|_G=|x|_G$, there also exists $f\in\mathrm{E}(A)$ with $f(x)=x+y'$, so it follows that $\pi f(x)=y'$, where $\pi:A\to G_2$ is the corresponding projection. If now $\varphi(y')=y$, then $\varphi\pi f(x)=y$, as required.
\end{proof}

Exploiting \cite{KMT}, there exists an abundance of torsion-free transitive non fully transitive groups, so that we can deduce:

\medskip

\begin{corollary}\label{2} There are such torsion-free transitive groups $G$ for which $G^{(\mathfrak{m})}$
is not transitive, provided $\mathfrak{m}>1$.
\end{corollary}

Firstly, we modify for our convenience and use the proof of \cite[Lemma 1]{FG} to torsion-free groups, stated there for $p$-groups.

\medskip

\begin{lemma}\label{3} If $G=A\oplus B$ is a fully transitive torsion-free group and $x=a+b$, $y=a_1+b_1$, where
$a,a_1\in A$, $b,b_1\in B$ with $|a|_A\leq |b_1-b|_B$, $|b_1|_B\leq |a_1-a|_A$, then $\varphi(x)=y$ for some
$\varphi\in\mathrm{Aut}(G)$.
\end{lemma}

\begin{proof} By assumptions, it must be that $\alpha(a)=b_1-b$ and $\beta(b_1)=a_1-a$ for some $\alpha,\beta\in\mathrm{E}(G)$. Therefore, one verifies that $$\varphi=\begin{pmatrix} \begin{array}{ccc} 1+\beta\alpha & \beta \\
\alpha & 1 \\
\end{array}
\end{pmatrix}\in\mathrm{Aut}(G)$$
and that
$\varphi(x)=y$, as required.
\end{proof}

So, we are ready to establish the following surprising assertion.

\begin{proposition}\label{4} If $G=\prod_{i\in I}G_i$ is a vector torsion-free group, where all direct factors $G_i$
are of rank $1$ with equal type, then $G$ is both a fully transitive and transitive group.
\end{proposition}

\begin{proof} If the index set $I$ is finite, then $G$ is a homogeneous fully decomposable group, so it is simultaneously fully transitive and transitive.

Letting now $I$ be an infinite index set, by virtue of \cite[Exersize~4 from \S 40]{KMT} (see also the original source \cite{GM}) the group $G$ is fully transitive.

However, we claim that it is also transitive. To that goal, suppose $x,y\in G$, $|x|_G=|y|_G$ and $G=A\oplus B$, where $A=\prod_{i\in I_1}G_i$, $B=\prod_{i\in I_2}G_i$, $I=I_1\cup I_2$, $I_1\cap I_2=\emptyset$ and $|I_1|=|I_2|$. Thus, one sees that $x=a+b$ and $y=a_1+b_1$, where $a,a_1\in A$ and $b,b_1\in B$. Let us now $x'=a'+b'$, where $a'\in A$,
$b'\in B$ and $|a'|_A=|b'|_B=|x|_G$. We, consequently, deduce that $|a'|_A\leq \mathrm{min}(|b-b'|_B,|b_1-b'|_B)$ and $|b'|_A\leq \mathrm{min}(|a-a'|_A,|a_1-a'|_A)$. According to Lemma~\ref{3}, one deduces that $\varphi(x')=x$ and $\psi(x')=y$ for some $\varphi,\psi\in\mathrm{Aut}(G)$. Hence, one extracts that $\psi\varphi^{-1}(x)=y$, as wanted.
\end{proof}

Let us recall once again for the readers' convenience that if $G$ is an arbitrary torsion-free group, then we define $\pi(G)=\{p\in\mathcal{P}\,|\,pG\neq G\}$, where $\mathcal{P}$ is the set of all prime numbers.

The following claim is a reminiscent of a already well-known result, but is listed here for our further use.

\begin{proposition}\label{5} (\cite[Proposition~1.3]{D}) If $G=\prod_{i\in I}G_i$ is a torsion-free group, where all
direct factors $G_i$ are such quasi-homogeneous transitive groups that $\pi(G_i)\cap\pi(G_j)=\emptyset$, then
$G$ is transitive.
\end{proposition}

The next statement is related to \cite[Problem 45]{KMT} and is a non-trivial refinement of Proposition~\ref{4}.

\begin{theorem}\label{6} The vector torsion-free group $G$ is transitive if, and only if, $G=\prod_{i\in I}G_i$,
where $\pi(G_i)\cap\pi(G_j)=\emptyset$ for any $i\neq j$, and every component $G_i$ is a direct product of rank $1$ groups of equal type.
\end{theorem}

\begin{proof} If $G$ assumed to be transitive, then it follows directly from \cite[Theorem~3.1]{D} that $\pi(A)\cap\pi(B)=\emptyset$ for any two  direct summands $A$ and $B$ of $G$ having rank $1$ and non equal type, as desired.

The converse implication follows immediately from a simple combination of Propositions~\ref{4} and ~\ref{5}.
\end{proof}

The next two technicalities shed some further light on the transversal between transitivity and full transitivity.

\begin{proposition}\label{7} If $G$ is a decomposable homogeneous fully transitive torsion-free group, then
$G$ is transitive.
\end{proposition}

\begin{proof} Let $G=A\oplus B$ with  $|x|_G=|y|_G$ and $x=a+b$, $y=a_1+b_1$, where $a,a_1\in A$ and $b,b_1\in B$.
Assume also that $x'=a'+b'$, where $a'\in A$, $b'\in B$ and $|a'|_A=|b'|_B=|x|_G$. Consequently, one sees that
$|a'|_A\leq \mathrm{min}(|b-b'|_B,|b_1-b'|_B)$ and $|b'|_A\leq \mathrm{min}(|a-a'|_A,|a_1-a'|_A)$. In accordance with Lemma~\ref{3}, one derives that $\varphi(x')=x$ and $\psi(x')=y$ for some $\varphi,\psi\in\mathrm{Aut}(G)$. Thus $\psi\varphi^{-1}(x)=y$, as desired.
\end{proof}

The following technical claim is closely related to Problem 41 (2) from \cite{KMT} and gives a partial affirmative answer to it.

\begin{lemma}\label{8} If $G$ is a homogeneous torsion-free transitive group, then $G^{(\alpha)}$ is transitive for all ordinals $\alpha \geq 1$.
\end{lemma}

\begin{proof} It is well known that any homogeneous transitive group is fully transitive (see, e.g., \cite[Lemma~25.2, Lemma 40.1]{KMT}). So, $G^{(\alpha)}$ is fully transitive with the aid of \cite[Exersice~12 from \S25]{KMT}.
Letting now $G^{(\alpha)}=G\oplus A$, where $A\cong G^{(\alpha-1)}$, it must be that $|x|_G=|y|_G$, $x=g+a$, $y=g'+a'$, $g,g'\in G$ and $a,a'\in A$. Also, let $g''\in G$, $a''\in A$ and $|g''|_G=|a''|_G=|x|_G$. Then, one verifies that $|g''|_G\leq \mathrm{min}(|a-a''|_G,|a'-a''|_G)$ and $|a''|\leq \mathrm{min}(|g-g''|_G,|g'-g''|_G$. So, Lemma~\ref{3} is a guarantor that $\varphi(g''+a'')=x$ and $\psi(g''+a'')=y$ hold for some $\varphi,\psi\in\mathrm{Aut}(G)$. Finally, it is a routine technical exercise to verify that $\psi\varphi^{-1}(x)=y$, and so we are done.
\end{proof}

\begin{remark}\label{comment} It follows also from \cite[Lemma~25.2]{KMT} that full transitivity and Krylov transitivity are equivalent properties for homogeneous torsion-free groups.
\end{remark}

What we can next offer are the following two statements:

\begin{corollary}\label{9} If $G$ is a homogeneous torsion-free group, then $G$ is fully transitive if, and only if,
$G^{(\alpha)}$ is transitive for some ordinal $\alpha>1$ (equivalently, for all ordinals $\alpha>1$).
\end{corollary}

\begin{proof} If $G^{(\alpha)}$ is transitive for some ordinal $\alpha>1$, then $G$ is fully transitive owing to Lemma~\ref{1}. The converse implication follows directly from Proposition~\ref{4}.
\end{proof}

\begin{proposition}\label{10} If $G$ is a homogeneous torsion-free transitive group, then the product $\prod_{\alpha}G$ is transitive for all ordinals $\alpha>1$.
\end{proposition}

\begin{proof} If the ordinal $\alpha$ is finite, then the assertion follows immediately from Lemma~\ref{8} since the product is isomorphic to the direct sum, while for infinite ordinals $\alpha$ we may apply the idea proposed in Proposition~\ref{4}.
\end{proof}

\begin{remark} As already commented above, the group in Proposition~\ref{10} is also fully transitive, so that the group $\prod_{\alpha}G$ is also fully transitive by exploiting \cite[Exersize~4 from \S40]{KMT}.
\end{remark}

It is worthwhile noticing that the direct summand of a (torsion-free) Krylov transitive group retains the same property. However, we will now improve this observation in the following quite non-trivial form as follows:

\begin{proposition}\label{11} $\mathrm{(1)}$ If we have a Krylov transitive torsion-free group $G\cong A^{(\mathfrak{m})}$ or $G\cong \prod_{\mathfrak{m}} A$ for some group $A$ with infinite ordinal $\mathfrak{m}$, then $G$ is fully transitive.

$\mathrm{(2)}$ If we have a Krylov transitive torsion-free group $G\cong A^{(m)}$ for some group $A$, where $m>1$ is a finite ordinal, then the group $A^{(n)}$ is fully transitive for every number $n$ with $1\leq n<m$ and, in particular, the group $A$ is fully transitive as well.
\end{proposition}

\begin{proof} (1) Since the ordinal $\mathfrak{m}$ is infinite, one writes that $G=B\oplus C$, where $G\cong B,C$. Applying Lemma~\ref{1}, both groups $B$ and $C$ are fully transitive, so $G$ is also fully transitive, as asserted.

(2) Let us write $$G=(A_1\oplus\dots\oplus A_n)\oplus (A_{n+1}\oplus\dots\oplus A_m),$$
where for all indices $i\in [1,m]$ the isomorphism $A_i\cong A$ holds, and $$|b=b_1+\dots+b_n|_G\geq |a=a_1+\dots+a_n|_G$$ with $a_j,b_j\in A_j$, $j=1,\dots,n$. If for a moment $f_j(a)=b_j$, then $(f_1+\dots+f_n)(a)=b$ holds, thus showing that $f_j$ really exists. Choosing $a_{n+1}\in A_{n+1}$ such that $|a_{n+1}|_G=|b_1|_G$, we deduce that $|a_1+\dots+a_n|_G=|a_1+\dots+a_n+a_{n+1}|_G$ and hence $f(a_1+\dots+a_n)=a_1+\dots+a_n+a_{n+1}$ holds for some $f\in\mathrm{E}(G)$. If, however, $\pi:G\to A_{n+1}$ is the corresponding projection, then it follows at once that $\pi f(a_1+\dots+a_n)=a_{n+1}$, whence the element $a_{n+1}$ will endomorphically be translated in the element $b_1$, as required.

The second part is now an immediate consequence of the first one.
\end{proof}

Let us notice that, in the case of finite ordinals, point (2) from the last assertion somewhat improves Lemma~\ref{1} quoted above.

\medskip

The following technical claim is at all similar to Theorem~\ref{6}, but we include it here without an explicit proof only for completeness of the exposition and for the reader's convenience.

\begin{proposition}\label{12} Let $G_i$, $i\in I$, be a set of homogeneous transitive groups $G_i$. Then the group $G=\prod_{i\in I}G_i$ is transitive if, and only if, $G=\prod_{j\in J}A_j$, where every
$A_j$ is a product of some direct components $G_j$ of the same type, and $\pi(A_j)\cap\pi(A_s)=\emptyset$ provided $j\neq s$, $j,s\in J$.
\end{proposition}

We, however, can reformulate this statement in an equivalent and more compact form and give it a complete proof as follows:

\medskip

{\it Let $\{G_i\}_{i\in I}$, where $I$ is an index set, be a family of homogeneous transitive groups $G_i$. Then $G=\prod_{i\in I}G_i$ is transitive if, and only if, $\pi(G_i)\cap\pi(G_i)=\emptyset$, provided $t(G_i)\neq t(G_j)$, for all $i,j\in I$}.

\medskip

\begin{proof} In view of Propositions~\ref{5} and \ref{10}, it is sufficient to show that for every direct summand $A\oplus B$ of the group $G$ with homogeneous components $A$, $B$ the condition $t(A)\neq t(B)$ implies $\pi(A)\cap\pi(B)=\emptyset$. To that goal, assume that $p\in\pi(A)\cap\pi(B)$, that either $t(B)>t(A)$ or $t(B)$, $t(A)$ are incomparable, and that $a\in A\setminus pA$, $b\in B\setminus pB$. Therefore, $|pa+b|_G=|a+b|_G$ and hence $f(pa+b)=a+b$ holds for some $f\in\mathrm{Aut}(G)$. If now $\pi:G\to A$ is the corresponding projection, then $p\pi f(a)+\pi f(b)=a$. But the equality $\pi f(B)=0$ follows in view of our assumption on types $t(A)$, $t(B)$, so that we obtain $p\pi f(a)=a$ which, however, contradicts the condition $a\in A\setminus pA$, as required.
\end{proof}

About an analogous variant of Proposition~\ref{12} related to direct sums of transitive and fully transitive homogeneous group, we refer the interested reader for more detailed information to the sources \cite[Exersize~12 from \S 25; Lemma~42.1, Exersize~6 from \S 40; Exersize~3 from \S 42]{KMT}).

\medskip

As an immediate consequence, we derive the following:

\begin{corollary}\label{13} The vector torsion-free Krylov transitive group is always fully transitive.
\end{corollary}

\medskip

If $t$ is the type of a torsion-free group $A$ of rank $1$, then for shortness of the records we write $\pi(t)=\pi(A)$.

\medskip

So, we are ready to proceed by proving the following assertion. It somewhat enlarge \cite[Theorems 2.4, 2.13]{DG} in the torsion-free situation and is also in sharp contrast to Example~\ref{(4)} from the next section in which a Krylov transitive non-fully transitive group is successfully constructed.

\begin{proposition}\label{14} Let $A$ be such a torsion-free group that either the types of all its non-zero elements are comparable, or if $t_1$ incomparable with $t_2$ then $\pi(t_1)\cap\pi(t_2)=\emptyset$. Then the Krylov transitivity of the group $G=A^{(\mathfrak{m})}$ for finite or infinite ordinal $\mathfrak{m}>1$ implies the full transitivity of $G$.
\end{proposition}

\begin{proof} In virtue of Proposition~\ref{11} (1), it suffices to consider the case of finite $\mathfrak{m}=m$. To this purpose, let we have $G=A_1\oplus\dots\oplus A_m$, where all $A_i\cong A$ and, as in Proposition~\ref{11}, the inequality $|b_1|_G\geq |a=a_1+\dots+a_m|_G$ is valid for all $a_i\in A_i$, $b_1\in A_1$. If some $t(a_j)$ and $t(a_s)$ are incomparable for $j\neq s$, then by condition it must be that $\pi(t(a_j))\cap\pi(t(a_s))=\emptyset$, so if now $a_j'\in A_j$ is an element such that $|a_j'|_G=|a_s|_G$, then $|a_j+a_j'|_G=|a_j+a_s|_G$. Therefore, if we replace the element $a_s$ by $a_j'$ in the decomposition of the element $a$, then we can get the element $a'$ with the property $|a|_G=|a'|_G$. So, we have that $\varphi(a)=a'$ for some $\varphi\in\mathrm{E}(G)$, the endomorphism ring of $G$. If, for a moment, $s\neq i$, then we see that the elements $a_i$ and $a'$ are contained in $A^{(m-1)}$, whence by induction $\psi(a')=b_i$ for some $\psi$ which obviously could be consider as an endomorphism of $G$. If, however, $s=i$, then in the group $A_j$ will exist an element $a_j'$ and elements $b_j',b_i$ with the property that $|b_j'|_G=|b_i|_G$ and, in order to complete our argument, it remains to take into account that there exists an endomorphisms of the group $G$ which translates the element $b_j'$ into the element $b_i$.

Now, in order to substantiate our initial claim, assume by hypothesis that all $t(a_i)$ are comparable and that $t(a_j)$ is less than the given type. Then $|a_j|_G\leq |kb_1|_G$ surely holds for some minimal natural number $k$. However, if now $a_j=ka'_j$ for some $a'_j\in A_j$, then $$|a_1+\dots+a_{j-1}+a'_j+a_{j+1}+\dots+a_m|_G=|a|_G$$ (note that the index $j$ could coincide with $1$ or with $m$, but this definitely will not affect our argumentation by changing the basic record accordingly if necessary). Furthermore, by translating the element $a$ endomorphically into the element $a_1+\dots+a_{j-1}+a'_j+a_{j+1}+\dots+a_m$ and projectively on the element $a'_j$, we next observe that we can translate the element $a'_j$ into the element $b_1$, as needed.
\end{proof}

Before proceeding further, we would like to mention explicitly the following important result established in \cite{FG}: {\it The Abelian $p$-group $A$, where $p$ is an arbitrary but fixed prime, is fully transitive if, and only if, the square $A\oplus A$ is transitive if, and only if, the square $A\oplus A$ is fully transitive. In addition, for any ordinal $\lambda >1$, the group $A^{(\lambda)}$ is fully transitive precisely when it is transitive}. Some parallel generalizations concerning Krylov transitivity can be found in \cite[Theorem 2.13]{DG}, too.

\medskip

In view of this nice criterion, we will be now concentrated on the case of $m=2$ in Lemma~\ref{1} and Proposition~\ref{11} (2). So, by collecting some of the ideas presented above combined with some new arguments presented in the subsequent section, we will expand both the aforementioned theorem of Files-Goldsmith to its analogue in the torsion-free situation and Corollary~\ref{2} as follows:

\begin{theorem}\label{15} If $A$ is a group such that the direct sum $A\oplus A$ is Krylov transitive (in particular, is both transitive and fully transitive), then $A$ is fully transitive. Conversely, there are groups $G$ which are simultaneously transitive and fully transitive, but such that their square $G\oplus G$ is not longer a Krylov transitive group.

In addition, there is a proper Krylov transitive group $G$ (which is neither transitive nor fully transitive) such that its square $G\oplus G$ is also not longer a Krylov transitive group.
\end{theorem}

\begin{proof} Given $A\oplus A$ is Krylov transitive for some group $A$. We now could directly apply either Lemma~\ref{1} and Proposition~\ref{11} (2) to get the wanted claim, but we shall now use a more transparent approach in proving that $A$ is necessarily fully transitive. In fact, knowing that the square $A_1\oplus A_2$ with $A_1,A_2\cong A$ is Krylov transitive, we will show that the group $A_1$ is fully transitive. In fact, let us choose $a,b\in A _1$ such that $|a|_A=|b|_A$. Letting also $c\in A_2$ with $|c|_A=|b|_A$, we then derive elementarily that $|a+c|_A=|a|_A$, and so $f(a)=a+c$, where $f\in \mathrm{E}(A_1\oplus A_2)$. Given  $\pi:A_1\oplus A_2\to A_1$ be the corresponding projection, we then deduce that $(1-\pi)(a+c)=c$, and hence $(1-\pi)f(a)=c$. However, by condition, we have that $g(c)=b$ for some $g\in \mathrm{E}(A_1\oplus A_2)$. Therefore, $\pi(g(1-\pi)f)\pi\in \mathrm{E}(A_1)$ whence $(\pi(g(1-\pi)f)\pi)(a)=b$, as required.

\medskip

As for the invalidity of the converse implication, we shall exhibit in what follows two independent constructions of groups which are both transitive and fully transitive having a few exotic properties like these:

\medskip

{\bf First Construction}: In \cite[Example 7.1]{Du} was constructed a transitive, fully transitive quasi-homogeneous group $G$ of finite rank, which is not homogeneous and in which every type of $\tau(G)$ is maximal (see also \cite[\S~41]{KMT}). Moreover, all endomorphisms of $G$ are just monomorphisms. The structure of $G$ entirely relies on a certain commutative extension of $\mathbb{Z}$ whose ideals are generated by integers (e.g., the most simple examples are the subrings of the ring of integers $\mathbb{Q}$).  

Now, suppose that $H=A_1\oplus A_2$, where $A_1\cong A_2\cong G$. Let us choose $a\in A_1\setminus pA_1$ and $b\in A_2\setminus pA_2$ with nonequal types. Therefore, one has that $|pa+b|_G=|a+b|_G$. Assuming then in a way of contradiction that $f(pa+b)=a+b$ for some $f\in\mathrm{E}(H)$, we will obtain that $pf(a)-a=b-f(b)\neq 0$ since $a\notin pA_1$. Furthermore, in view of maximality of the types, one may calculate that $t(pf(a)-a)=t(a)$ and $t(b-f(b))=t(b)$. So, the equality $pf(a)-a=b-f(b)\neq 0$ is absolutely impossible, that is the wanted contradiction thus substantiating our claim.

\medskip

{\bf Second Construction}: We claim here that {\it there exists a group $G$ of rank $3$ that is both transitive and fully transitive such that $G\oplus G$ need not be Krylov transitive. Even more, such a group can be constructed whose endomorphism ring is isomorphic to $\mathbb Z$}.

\medskip

In fact, the following construction is clearly closely related to the construction in Example~\ref{(3)} quoted below in the subsequent section. We will construct the desired group $G$ so that

$$F:=\langle\bbvec a_1 \rangle\oplus\langle\bbvec a_2 \rangle\oplus\langle\bbvec a_3\rangle\subseteq G\subseteq {\mathbb Q} F:= V.$$

\medskip

Extend $\{\bbvec a_1, \bbvec a_2, \bbvec a_3\}$ to a list $\{\bbvec a_k: k=1, 2, \cdots\}$ so that $\{A_k:= \langle\bbvec a_k\rangle: k=1,2,\cdots\}$ is a complete list of all rank one summands of $F$.

Next, for each $k=1,2,\dots$, let $F=A_k\oplus B_k$ be a decomposition.

We let ${\mathcal P}_k$ for $k=1,2,\dots$ be a collection of infinite disjoint sets of primes (in particular, ${\mathcal P}_1$ is not empty as it is demonstrated in the cited Example~\ref{(3)}). For each $k$, after possibly throwing out a finite number of primes in ${\mathcal P}_k$, we may assume that these sets satisfy the following pivotal condition:

\medskip

$(\dag)$ {\it If $j=1,2,3$, $k\in \{1,2,\dots\}$, $j\ne k$ and $\bbvec a_j=a_{j,k}+b_{j,k}\in A_k\oplus B_k$, then $\val{b_{j,k}}_{F,p}=0$ for all $p\in {\mathcal P}_k$.}

\medskip

Again as in Example~\ref{(3)}, we define
$$
              G:= F+\langle (1/p)\bbvec a_k: k=1,2,\dots, {\rm and}\ p\in {\mathcal P}_k\rangle.
$$

\medskip

The following facts are established using totally analogous arguments to those found in Example~\ref{(3)}:

\medskip

{\bf (A)} The types of $\bbvec a_k$ in $G$ for distinct $k=1,2,\dots$ are not comparable. (Note that in Example~\ref{(3)} itself, $\chi(\bbvec a_1)={\bf 0}$ was actually strictly smaller than the types of any other $\bbvec a_k$, for $k>1$.)

\medskip

{\bf (B)} The endomorphism ring of $G$ is indeed isomorphic to $\mathbb Z$ (this uses an even simpler version of Lemma~\ref{first}).

\medskip

{\bf (C)} If $j=1,2,3$, then condition $(\dag)$ implies that
$$
 \val {\bbvec a_j}_{G,p}=\begin{cases}
                          1, & {\rm if}\ p\in {\mathcal P}_j;\cr
                          0, & {\rm otherwise.}\cr
\end{cases}
$$

\medskip

Clearly, all $p$-heights of non-zero elements of $G$ will be finite. This observation, together with points (A) and (B), is completely enough to show that $G$ is both transitive and fully transitive.

It now follows from (C) that
$$
          \chi({\bbvec a}_1)\wedge \chi({\bbvec a}_{2})=\chi({\bbvec a}_1)\wedge \chi({\bbvec a}_{3})=(0,0,\dots)={\bf 0}.
$$

\medskip

We need to verify that $G\oplus G$ fails to be Krylov transitive. To that end, we use the above letters $\bbvec a_j$ to represent the above elements of $G$ thought of as in the first summand of $G\oplus G$, and we use $\bbvec b_j$ to represent the corresponding elements of $G$ thought of as in the second summand of $G\oplus G$.

In $G\oplus G$, consider $x:= \bbvec a_1+ \bbvec b_2$ and $y:=\bbvec a_1+\bbvec b_3$. It easily follows that $\chi(x)=\chi(y)={\bf 0}$. However, if $\phi$ is an endomorphism of $G\oplus G$, it follows easily from (B) that
$$
              \phi(x) \in \langle \bbvec a_1,\bbvec a_2, \bbvec b_1,\bbvec b_2\rangle.
$$

\medskip

\noindent This immediately implies that $\phi(x)\ne y$, so that $G\oplus G$ is not Krylov transitive, as promised.

\medskip

Next, to show the truthfulness of the additional part, with Example~\ref{(4)} from the next section at hand, the existence of a Krylov transitive group which is neither transitive nor fully transitive helps us to employ the first part of the statement to conclude the desired claim.

Nevertheless, as a totally different and more direct argument in showing that, we may propose the following one:

\medskip

Let $B$ be a group which is fully transitive, but not transitive (see, e.g., Example~\ref{(2)} below). Also, let $A$ be a group which is transitive, but not fully transitive (see, e.g., Example~\ref{(3)} below), making sure that, for every prime $p$, one of the groups $B$ and $A$ is $p$-divisible. We then let $G=A\oplus B$ and it can be shown that it possesses the requisite properties. So, assume that is the case and, moreover, that $\bbvec a_1$, $\bbvec a_2$ and $\bbvec a_3$ are linearly independent elements of $F$ that each generate summands of $F$ that operate as in the proof of Example~\ref{(3)}. Furthermore, if the squares $G\oplus G$ were Krylov transitive groups, it would immediately imply that its summand $A\oplus A$ would also be Krylov transitive. With some very minor changes, the above argumentation again leads to a contradiction:

\medskip

Indeed, as showed above, think of $\bbvec b_1$, $\bbvec b_2$ and $\bbvec b_3$ being the corresponding elements of $A\oplus A$ in the second summand. And again, let we put $x:= \bbvec a_1+ \bbvec b_2$ and $y:=\bbvec a_1+\bbvec b_3$. It easily follows now that $\chi(x)=\chi(y)=\chi({\bbvec a_1})$. However, if $\phi$ is any endomorphism of $A\oplus A$, it follows at once that
$$
              \phi(x) \in \langle \bbvec a_1,\bbvec a_2, \bbvec b_1,\bbvec b_2\rangle.
$$
This, in turn, immediately implies that $\phi(x)\ne y$, so that $A\oplus A$ is not Krylov transitive which is the wanted contrary to our assumption.
\end{proof}

We finish off this section with the following commentaries:

\begin{remark} It is worthwhile noticing that the last assertion unambiguously shows that there is no an absolute analog with the cited above theorem of Files-Goldsmith from \cite{FG}, because the torsion-free case is rather more complicated than the primary case.
\end{remark}

\section{Constructions and Examples}

The relationships between our four types of ``transitivity" are illustrated in the following {\bf Venn diagram}:

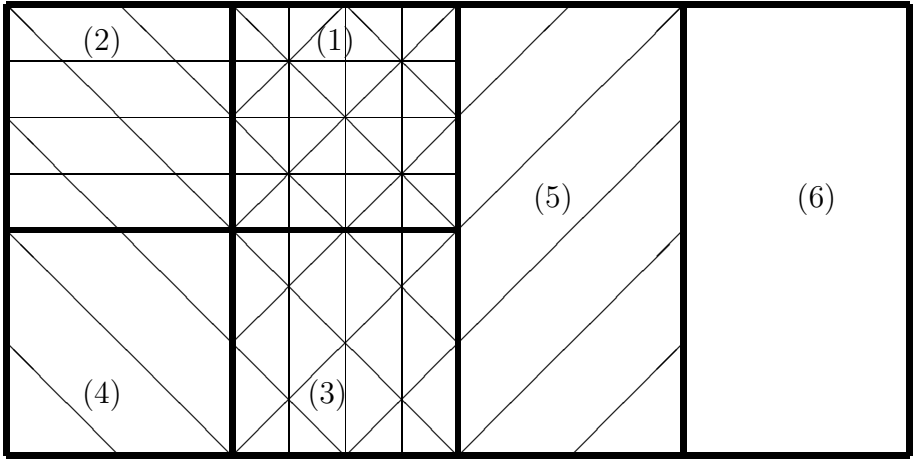
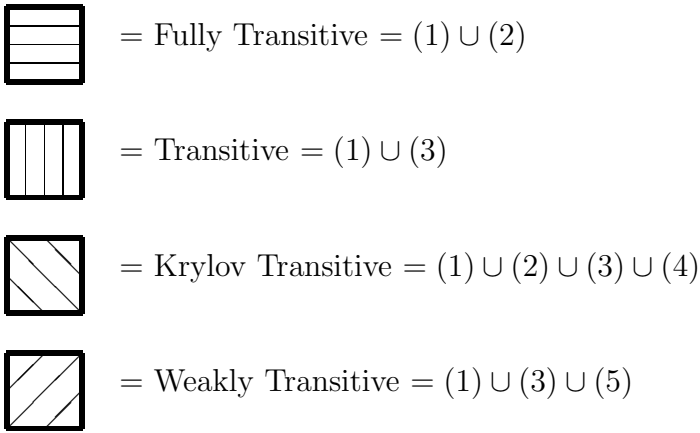
\setlength{\unitlength}{1cm}

\begin{picture}(14,7)
\linethickness{.7mm}
\put(0,0,0){\line(1,0){12}}
\put(0,3,0){\line(1,0){6}}
\put(0,6,0){\line(1,0){12}}
\put(0,0,0){\line(0,01){6}}
\put(3,0,0){\line(0,01){6}}
\put(6,0,0){\line(0,01){6}}
\put(9,0,0){\line(0,01){6}}
\put(12,0,0){\line(0,01){6}}
\linethickness{.05mm}
\put(0,3.75){\line(1,0){6}}
\put(0,4.5){\line(1,0){6}}
\put(0,5.25){\line(1,0){6}}
\put(3.75,0){\line(0,1){6}}
\put(4.5,0){\line(0,1){6}}
\put(5.25,0){\line(0,1){6}}
\put(1.5,0){\line(-1,1){1.5}}
\put(3,0){\line(-1,1){3}}
\put(4.5,0){\line(-1,1){4.5}}
\put(6,0){\line(-1,1){6}}
\put(6,1.5){\line(-1,1){4.5}}
\put(6,3){\line(-1,1){3}}
\put(6,4.5){\line(-1,1){1.5}}
\put(3,0){\line(1,1){6}}
\put(3,1.5){\line(1,1){4.5}}
\put(3,3){\line(1,1){3}}
\put(3,4.5){\line(1,1){1.5}}
\put(4.5,0){\line(1,1){4.5}}
\put(6,0){\line(1,1){3}}
\put(7.5,0){\line(1,1){1.5}}
\put(1,5.4){(2)}
\put(4.1,5.4){(1)}
\put(4.0,.7){(3)}
\put(1.0,.7){(4)}
\put(7,3.3){(5)}
\put(10.5,3.3){(6)}
\end{picture}

\vskip .2in

\setlength{\unitlength}{.5cm}
\begin{picture}(14,3)
\linethickness{.6mm}
\put(0,0,0){\line(1,0){2}}
\put(0,0,0){\line(0,1){2}}
\put(0,2,0){\line(1,0){2}}
\put(2,0,0){\line(0,1){2}}
\put(3,1,0){= Fully Transitive $=(1)\cup (2)$}
\linethickness{.05mm}
\put(0,.5){\line(1,0){2}}
\put(0,1.0){\line(1,0){2}}
\put(0,1.5){\line(1,0){2}}
\end{picture}

\begin{picture}(14,3)
\linethickness{.6mm}
\put(0,0,0){\line(1,0){2}}
\put(0,0,0){\line(0,1){2}}
\put(0,2,0){\line(1,0){2}}
\put(2,0,0){\line(0,1){2}}
\put(3,1,0){= Transitive $=(1)\cup (3)$}
\linethickness{.05mm}
\put(.5,0){\line(0,1){2}}
\put(1.0,0){\line(0,1){2}}
\put(1.5,0){\line(0,1){2}}
\end{picture}

\begin{picture}(14,3)
\linethickness{.6mm}
\put(0,0,0){\line(1,0){2}}
\put(0,0,0){\line(0,1){2}}
\put(0,2,0){\line(1,0){2}}
\put(2,0,0){\line(0,1){2}}
\put(3,1,0){= Krylov Transitive $=(1)\cup (2)\cup (3)\cup (4)$}
\linethickness{.001mm}
\put(1,0){\line(-1,1){1}}
\put(2.0,0){\line(-1,1){2}}
\put(2,1){\line(-1,1){1}}
\end{picture}

\begin{picture}(14,3)
\linethickness{.6mm}
\put(0,0,0){\line(1,0){2}}
\put(0,0,0){\line(0,1){2}}
\put(0,2,0){\line(1,0){2}}
\put(2,0,0){\line(0,1){2}}
\put(3,1,0){= Weakly Transitive $= (1)\cup (3)\cup (5)$}
\linethickness{.001mm}
\put(1,0){\line(1,1){1}}
\put(0,0){\line(1,1){2}}
\put(0,1){\line(1,1){1}}
\end{picture}

\vskip3.0pc

This diagram unambiguously illustrates, for instance, that the class of transitive groups is the intersection of the class of weakly transitive groups with the class of Krylov transitive group. We claim that the indicated disjoint regions (1)-(6) are non-empty; that is, each of which contains some group.

\medskip

We continue with a series of examples by constructing groups lying in all of the indicated above regions.

\begin{example}\label{(1)} There is a group in region (1), i.e., it is both transitive and fully transitive (and hence also weakly and Krylov transitive).
\end{example}

\begin{proof} Let $G$ be any rank-$1$ group.
\end{proof}

\begin{example}\label{(2)} There is a group in region (2), i.e., it is fully transitive (and hence Krylov transitive), but not transitive (and hence not weakly transitive).
\end{example}

\begin{proof} This was referred to on page 385 of \cite{KMT}. In particular, this group $G$ has the property that $G/pG \cong {\mathbb Z}_p$ for each prime $p$.
\end{proof}

\begin{example}\label{(3)} There is a group in region (3), i.e., it is transitive (and hence both Krylov and weakly transitive), but not fully transitive.
\end{example}

\begin{proof} The existence of such a group is also referred to on page 385 of \cite{KMT}, but however we found another way to construct an example which has numerous exotic properties. Firstly, we use the following elementary but very useful observation:

\begin{lemma}\label{first} Suppose $V$ is a ${\mathbb Q}$-vector space, $\phi:V\to V$ is a linear transformation and $v\in V$. If for every $w\not\in {\mathbb Q}v$ we have $\phi(w)\in {\mathbb Q}w$, then there is a single $\alpha\in {\mathbb Q}$ such that $\phi(w)=\alpha w$ for all $w\in V$.
\end{lemma}

\begin{proof} First and foremost, we shall explain why the defined linear transformations are monomorphisms: In fact, we may clearly assume $V\ne \mathbb{Q} v$. Now, for every $w\not\in \mathbb{Q} v$, there is a scalar $\alpha_w$ such that $\phi(w)=\alpha_w w$; so if $z\in \mathbb{Q} w$ we also have $\phi(z)=\alpha_w(z)$.

Next, if $z\not\in \mathbb{Q} v \cup  \mathbb{Q} w$, then either $w+z$ or $w+2z$ is not in $\mathbb{Q} v$. Replacing $z$ by $2z$ if necessary, we may assume $y:=w+z$ is not in $\mathbb{Q} v$. So, we must have

$$\alpha_y  y = \phi(y)=\phi(w)+\phi(z)=\alpha_w w + \alpha_z z.$$

\medskip

But since $w$ and $z$ are linearly independent, this implies that $\alpha_y=\alpha_w=\alpha_z$.  In other words, for all $z\not\in \mathbb{Q} v$, we must have $\alpha_z=\alpha_w$.

Denote this constant value by $\alpha$. Since $w$, $v-w$ are not in $\mathbb{Q} v$, we can conclude that

$$\phi(v)=\phi(v-w+w)=\alpha(v-w)+\alpha(w)= \alpha v.$$

\medskip

This gives the desired conclusion.

We next may clearly assume that $V$ does not have dimension $1$. Our hypotheses imply that every vector $w\not\in {\mathbb Q}v$ must be in some eigenspace for $\phi$. This apparently implies that there can be only one non-zero eigenspace for $\phi$ and that that eigenspace contains all of $V$, giving the result.
\end{proof}

Furthermore, returning to the former example, the wanted group $G$ we going to construct is transitive (and so both Krylov and weakly transitive) and not fully transitive, but it has some interesting additional properties like these:

\medskip

(1) If $F$ is any non-cyclic free group of countable rank, such a group $G$ can be found with $F\subset G\subset {\mathbb Q}F$;

(2) The endomorphism ring of $G$ is $\mathbb Z$, i.e., any endomorphism is simply multiplication by an integer;

(3) Any two pure rank-$1$ subgroups of $G$ have different types.

\medskip

The condition that a group is transitive means that, in some sense, it has ``enough" automorphisms for its size. The transitive group in this example, however, can be fairly ``large" (i.e., countably infinite rank), and still have an automorphism group that is as small as possible (namely $\pm 1$).

\medskip Turning to our construction, let $A_k$ for $k\in {\mathbb N}=\{1,2,3,\dots\}$ be the collection of all rank-1 pure subgroups of $F$ (i.e., rank-$1$ direct summands of $F$); and let ${\bf a}_k\in A_k$ be some generator of $A_k$.

For each $k\in \mathbb N$ fix a decomposition $F=A_k\oplus B_k$.  If $j\ne k$, then there are unique $a_{j,k}\in A_k$ and $b_{j,k}\in B_k$ such that ${\bf a}_j = a_{j,k}+b_{j,k}$. Since ${\bf a}_j$ and ${\bf a}_k$ determine different summands of $F$, we must have $b_{j,k}\ne 0$. In particular, $\val {b_{j,k}}_{F,p}=0$ for all but finitely many primes.

Let ${\mathcal P}_k$ for $k\in {\mathbb N}$ be a collection of disjoint sets of primes with the following properties:

(a) ${\mathcal P}_1=\emptyset$;

(b) If $k>1$, then ${\mathcal P}_k$ is infinite;

(c) If $k>1$ and $p\in {\mathcal P}_k$, then $\val {b_{1,k}}_{F,p}=0$.

\noindent  Regarding point (c), we just choose each of our infinite disjoint sets ${\mathcal P}_k$ for $k>1$ so as to avoid the finite number of primes where $\val {b_{1,k}}_{F,p}$ is non-zero.

Let
$$
          G=\langle  {\bf a}_k/p: k\in {\mathbb N}, p\in {\mathcal P}_k\rangle\subset {\mathbb Q}F.
$$
Further, for a prime $p$, let
$$
         F^p =
               \begin{cases}

                  F, &  {\rm if\ }p\not\in \cup_{k\in \mathbb N} {\mathcal P}_k.\cr
                  {\mathbb Z} ({\bf a}_k/p)\oplus B_k=(1/p)A_k\oplus B_k,   &{\rm if}\ p\in {\mathcal P}_k,\cr
               \end{cases}
$$
In other words,  for all primes $p$, $F^p/F$ will be the $p$-torsion subgroup of $G/F$.

Note that $G/F^p$ will have no $p$-torsion, so that if $x\in F$ then $\val x_{G,p}=\val x_{F^p,p}<\infty$. In particular, all $p$-heights of non-zero elements of $G$ will be finite.

\medskip

\medskip\noindent {\bf Claim~A:} $\chi({\bf a}_1)=(0,0,\dots):={\bf 0}$.

\medskip

Let $p$ be some prime. If $p\not\in \cup_{k\in \mathbb N} {\mathcal P}_k$, then
$$\val {{\bf a}_1}_{G,p}=\val {{\bf a}_1}_{F^p,p}=\val {{\bf a}_1}_{F,p}=0.
$$
And if $p\in {\mathcal P}_k$, then $k>1$ and by condition (c) in our choice of ${\mathcal P}_k$,
$$\val {{\bf a}_1}_{G,p}=\val {{\bf a}_1}_{F^p,p}= \val {a_{1,k}+b_{1,k}}_{F^p,p}=0.
$$

\medskip If $k>1$, then since $\val {{\bf a}_k}_{G,p}=1$ for all $p\in {\mathcal P}_k$, we have $\t({\bf a}_1)=\overline {\bf 0}<\t ({\bf a}_k)$.

\medskip

\medskip \noindent {\bf Claim B:} If $j,k$ are distinct elements of $\{2, 3, \dots\}$, then  $\t ({\bf a}_j)$ and $ \t ({\bf a}_k)$ are incomparable.

\medskip

For all $p\in {\mathcal P}_k$ we have $\val {{\bf a}_k}_{G,p}=1$ and except for finitely many such $p$ we have

$$\val {{\bf a}_j}_{G,p}=\val {{\bf a}_j}_{F^p,p}= \val {a_{j,k}+b_{j,k}}_{F^p,p}\leq \val {b_{j,k}}_{B_k,p}=0.
$$

\medskip

\noindent In particular, $\t ({\bf a}_k)\not \leq \t ({\bf a}_j)$; similarly, $\t ({\bf a}_j)\not \leq \t ({\bf a}_k)$.

\medskip

Let $V={\mathbb Q}F$. Any subspace of $V$ of dimension 1 is of the form ${\mathbb Q} {\bf a}_k$ for a unique $k\in \mathbb N$. Therefore, Claim~B implies that the elements of the typeset of $G$ are in one-to-one correspondence with the rank-1 subgroups of $G$.

\medskip

\medskip \noindent {\bf Claim C:} Every endomorphism of $G$ is multiplication by some integer $n$.

\medskip

Suppose $\phi:G\to G$ is an endomorphism; we can clearly extend this to an endomorphism $\phi: V\to V$. We apply Lemma~\ref{first} with $v={\bf a}_1$. If $k>1$, we need to show $\phi({\bf a}_k)\in {\mathbb Q} {\bf a}_k$.  If $\phi({\bf a}_k)=0$ it is trivially so, and if $\phi({\bf a}_k)\ne 0$, then $\phi({\bf a}_k)\in {\mathbb Q} {\bf a}_j$ for some unique $j\in \mathbb N$. Since
$$\t ({\bf a}_1)=\overline {\bf 0}< \t ({\bf a}_k)\leq \t (\phi({\bf a}_k))=\t ({\bf a}_j),
$$
we have $j>1$. And since $\t ({\bf a}_j)$ and $\t ({\bf a}_k)$ are comparable, by Claim~B we may conclude that $j=k$, so that $\phi({\bf a}_k)\in {\mathbb Q} {\bf a}_j = {\mathbb Q}{\bf a}_k$, as required.

Therefore, by Lemma~\ref{first} we can conclude that $\phi$ is multiplication by some rational number. But since all $p$-heights in $G$ are finite, this rational number must be an integer, as required.

\medskip  We can now quickly complete the proof of the example. To that goal, note that $\chi({\bf a}_1)={\bf 0}<\chi({\bf a}_2)$, but ${\bf a}_2\not\in A_1={\mathbb Z}{\bf a}_1$, so $G$ is clearly not fully transitive.

On the other hand, suppose $x,y\in G$ are non-zero with $\chi(x)= \chi(y)$. There are unique integers $j,k\in \mathbb N$ such that $x\in {\mathbb Q}{\bf a}_j$ and $y\in {\mathbb Q}{\bf a}_k$. It follows that
$$
     \t({\bf a}_j)=\t(x)=    \t (y)=\t ({\bf a}_k).
$$

\medskip

\noindent So, by Claim~B, $j$ must equal $k$. Therefore, all four of these elements come from the same pure rank-1 subgroup of $G$. But since all $p$-heights computed in $G$ are finite, this fact, together with $\chi(x)= \chi(y)$, implies that $x=\pm y$. In particular, $G$ must be transitive (and hence Krylov transitive).
\end{proof}

We now continue with the promised in the Abstract positive solution of \cite[Problem 44]{KMT}. Precisely, the following is true:

\begin{example}\label{(4)} There is a group in region (4), i.e., it is Krylov transitive, but neither fully transitive nor transitive (and hence not weakly transitive). In particular, this example answers in the affirmative the long-standing Problem 44 of \cite{KMT}.
\end{example}

\begin{proof} Partition the collection of all prime numbers into two infinite disjoint subsets, $P_1$ and $P_2$. Let $R_1$ be the integers localized at $P_1$ (so $R_1$ is $p$-divisible for all primes $p$ not in $P_1$) and $R_2$ be the integers localized at $P_2$. Since both $R_1$ and $R_2$ are countable infinite Euclidean domains with a countably infinite set of primes, it is clear that any construction that we can carry out for groups we can also carry out for modules over $R_1$ or $R_2$.

Therefore, we can construct a reduced $R_1$-module $A$ in the corresponding region (2) for $R_1$-modules, i.e., it is fully transitive but not transitive. Similarly, we can construct a reduced $R_2$-module $B$ in the corresponding region (3) for $R_2$-modules, i.e., it is transitive, but not fully transitive.

We now think of both $A$ and $B$ simply as groups and let $G=A\oplus B$.
It is also easy to see that if $x=(a,b)\in G$ and $y=(a',b')\in G$, then $\val x_G\leq \val y_G$ if and only if $\val a_A\leq \val {a'}_A$ and $\val b_B\leq \val {b'}_B$.

To show $G$ is Krylov transitive, suppose $\val x_G= \val y_G$. It follows that $\val a_A= \val {a'}_A$ and $\val b_B\leq \val {b'}_B$. Since $A$ is fully transitive, there is a endomorphism $\lambda:A\to A$ with $\lambda(a)=a'$, and since $B$ is transitive, there is an automorphism $\gamma:B\to B$ such that $\gamma(b)=b'$. Therefore, if $\phi=(\lambda, \gamma)$, then $\phi(x)=y$, as required.

Note that if $A\to B$ is a non-zero homomorphism, then the elements of its image would be divisible in $B$  by every $p\in P_1$ as well as by every prime $p\in P_2$. That would imply that $B$ is not reduced, contrary to its construction. So, there are no non-zero homomorphisms $A\to B$. Similarly, there are no non-zero homomorphism $B\to A$. So, $A$ and $B$ are individually fully invariant subgroups of $G$.

Since $A$ is not transitive there are elements $a,a'\in A$ with $\chi(a)=\chi(a')$ such that there is no automorphism of $A$ maps $a$ to $\a'$. It easily follows that no automorphism of $G$ maps $a$ to $a'$, so that $G$ also fails to be transitive.  Similarly, since $B$ is not fully transitive, it again follows that $G$ is not fully transitive, either.
\end{proof}

In regard to the last statement, which completely settles \cite[Problem 44]{KMT}, it is worth to notice that in \cite{BGGS} was constructed a Krylov transitive Abelian $2$-group which is neither fully transitive nor transitive. So, the last statement could be viewed as a torsion-free analogue of the cited construction.

\medskip

What we can now offer is the following:

\begin{example}\label{(5)} There is a group in region (5), i.e., it is weakly transitive, but not fully transitive nor  transitive (and hence not Krylov transitive).
\end{example}

\begin{proof} The following result clearly produces a large number of examples of groups that satisfy the above requirements, as promised.
\end{proof}

\begin{theorem}\label{problem}
Suppose that $A$ and $B$ are weakly transitive groups such that ${\rm Hom}(A,B)=0={\rm Hom}(B,A)$ and for which there is a prime $p$ such that neither $A$ nor $B$ is $p$-divisible. Then $G=A\oplus B$ is weakly transitive, but not  transitive nor fully transitive.
\end{theorem}

\begin{proof} The condition ${\rm Hom}(A,B)=0={\rm Hom}(B,A)$ implies that for any endomorphism $\xi$ of $G$, there is an endomorphism $\lambda$ of $A$ and an endomorphism $\gamma$ of $B$ such that for all $(a,b)\in G$ we have $\xi((a,b))=(\lambda (a), \gamma(b))$. We write $\xi=(\lambda, \gamma)$ in this case.

To show $G$ is weakly transitive, suppose $x=(a,b)$ and $y=(a',b')\in G$, and that there are endomorphism $\xi$ and $\xi'$ of $G$ such that $\xi(x)=y$ and $\xi'(y)=x$. If $\xi=(\lambda, \gamma)$ and $\xi'=(\lambda', \gamma')$, it follows that $\lambda(a)=a'$, $\lambda'(a')=a$, $\gamma(b)=b'$ and $\gamma'(b')=b$. Since $A$ and $B$ are weakly transitive, there are automorphisms $\lambda$ of $A$ and $\mu$ of $B$ such that $\lambda(a)=a'$ and $\mu (b)=b'$. If we let $\phi=(\lambda,\mu)$, then $\phi$ will be an automorphism of $G$ with $\phi(x)=y$. Therefore, $G$ is weakly transitive.

To show $G$ is not transitive, we use our condition on the prime $p$. Let $a\in A$ and $b\in B$ satisfy $\val a_A=\val b_B=0$. Consider the elements $x:=(pa,b)$ and $y:= (a,b)$ in $G$.  It is easy to see that $\chi_G(x)=\chi_G(y)$. Now, if $\phi=(\lambda, \mu)$ was an automorphism of $G$ with $\phi(x)=y$, then we could conclude that $\lambda$ is an automorphism of $A$ with $\lambda(pa)=a$. Since this is clearly impossible, $G$ fails to be transitive.

Note that if $G$ were to be fully transitive, then it is also Krylov transitive, and since $G$ is weakly transitive, it would have to be transitive. Since that is not the case, $G$ also fails to be fully transitive.
\end{proof}

For example, to be more concrete, if $A$ and $B$ are rank-$1$ groups of incomparable type and $p$ is a prime such that neither $A$ nor $B$ is $p$-divisible, then $G=A\oplus B$ satisfies Theorem~\ref{problem}; so there are rather small and simple groups that satisfy Problem~44. More generally, if $\alpha$ and $\beta$ are incomparable types, $A$ is  $\alpha$-homogeneous and transitive and $B$ is $\beta$-homogeneous and transitive, and neither $A$ nor $B$ is $p$-divisible, then again, $G=A\oplus B$ will satisfy the requirements. In other words, the group $G$ can also be quite large and complicated. The following assertion shows that the condition on the prime number $p$ is absolutely necessary and cannot be ignored.

\begin{example} Suppose $A$ and $B$ are rank $1$ groups and, for every prime $p$, exactly one of $A$ or $B$ is $p$-divisible. Then $G=A\oplus B$ is simultaneously transitive and fully transitive.
\end{example}

\begin{proof} Suppose $x=(a,b)$ and $y=(a',b')\in G$. Our condition on divisibility implies that $\chi_G(x)\leq \chi_G(y)$ if and only if $\chi_A(a)\leq \chi_A(a')$ and $\chi_B(b)\leq \chi_B(b')$. Since $A$ and $B$ are transitive and fully transitive, this readily implies that $G$ is so, as well.
\end{proof}

We now manage to prove our final example. Concretely, the following is valid:

\begin{example}\label{(6)} There is a group in region (6), i.e., it is neither Krylov transitive (and so not transitive or fully transitive), nor weakly transitive.
\end{example}

\begin{proof} Let $H\subseteq \mathbb Q$ be a rank-$1$ group with the following properties:

(1) $1\in H$ and $\val 1_{H,5}=0$;

(2) The endomorphism ring of $H$ equals $\mathbb Z$;

(3) ${\rm Hom}(H, {\mathbb Z})=0$, i.e., $H$ is not cyclic.

\medskip  Let $G={\mathbb Z}\oplus H$. Condition (3) above implies that $H$ is fully invariant in $G$. If we view the elements of $G$ as row vectors, then endomorphisms of $G$ can be viewed as right multiplications by matrices of the form:
$$M=
\left[\begin{matrix}
   m & h \cr
   0 & n \cr
\end{matrix}\right]
$$
where $m,n\in \mathbb Z$ and $h\in H$.

\medskip To verify that $G$ is not Krylov transitive, consider the vectors $x=(5,1)$ and $y=(1,1)$. It is fairly clear that $\val x_p=\val y_p=0$ for all primes $p$. However, if $\phi$ is any endomorphism of $G$, then we may assume $\phi$ is multiplication by the above matrix $M$. It follows that $\phi(x)=(5m, 5h+n)\ne (1,1)$. Therefore, $G$ is obviously not Krylov transitive, as pursued.

To show $G$ also fails to be weakly transitive, consider the elements $x=(5,1)$ and $y=(5,2)$. Note that
$$
    y = x \left[\begin{matrix}
   1 & 0 \cr
   0 & 2 \cr
\end{matrix}\right]\ \ \ {\rm and }\ \ \
x = y \left[\begin{matrix}
   1 & -1 \cr
   0 & 3 \cr
\end{matrix}\right].
$$

On the other hand, we claim that there is no automorphism $\phi$ of $G$ with $\phi(x)=y$. To verify this, note that if $\phi$ was such an automorphism and again, it is multiplication by the above matrix $M$, then it is easy to see that $m,n=\pm 1$. This would imply that
$$
           (5,2)=y=\phi(x)= (5,1)\left[\begin{matrix}
   \pm 1 & h \cr
   0 & \pm 1 \cr
\end{matrix}\right] =(\pm 5, 5h\pm 1).
$$
This would imply that in $H/5H\cong {\mathbb Z}_5$ we have that the element $[2]$ is equal to either $[1]$ or $[-1]=[4]$, which is evidently not the case.
\end{proof}

\section{Concluding Discussion and Open Problems}

A few additional comments concerning our results established above could be the following: First of all, we would like to point out that Lemma~\ref{1} could be extended for an arbitrary Abelian group (not necessarily torsion-free) in a rather non-trivial way as follows: {\it Given a Krylov transitive group $G=A\oplus B$ such that the complement $B$ contains a direct summand isomorphic to the group $A$, then $A$ is fully transitive}.

As a further remark, we indicate that for all our groups exhibited in the listed Example~\ref{(1)}, etc. , Example~\ref{(6)}, their square is always weakly transitive. As the work for proving this is rather technical and too hard, we voluntarily omit the details in the existing proofs, but will state below the corresponding Problem 3 to visualize the difficulty of the general situation.

\medskip

In closing, we pose three challenging questions of some interest and importance which directly arise.

\medskip

\noindent{\bf Problem 1.} Do there exist fully transitive non transitive homogeneous torsion-free groups?

\medskip

The next two queries are closely relevant to Theorem~\ref{15}.

\medskip

\noindent{\bf Problem 2.} For a group $G$ does the Krylov transitivity of the square $G\oplus G$ imply that $G$ is transitive?

\medskip

\noindent{\bf Problem 3.} If $G$ is either a Krylov transitive group or a weakly transitive group, respectively, is it true that the square $G\oplus G$ is always a weakly transitive group?

\medskip

%\noindent{\bf Acknowledgment.} ...

\medskip

\noindent {\bf Funding:} The work of the second named author P.V. Danchev is partially supported by the Bulgarian National Science Fund under Grant KP-06 No. 32/1 of December 07, 2019.

\end{document}